\documentclass[12pt]{amsart}
\usepackage[top=30truemm,bottom=30truemm,left=25truemm,right=25truemm]{geometry}
\usepackage{txfonts}
\usepackage{mathrsfs}

\usepackage{color}
\usepackage{bm}
\usepackage{amsfonts,amssymb}
\usepackage{dsfont}
\usepackage{amscd}
\usepackage{extarrows}
\usepackage{amsmath}
\usepackage{mathrsfs}
\usepackage{enumerate}
\usepackage{amscd}
\usepackage[all]{xy}
\usepackage[hyperfootnotes=true]{hyperref}

\usepackage{extarrows}
\theoremstyle{plain} 
\newtheorem{theorem}{\indent\bf Theorem}[section]

\newtheorem{conj}[theorem]{\indent\bf Conjecture}
\theoremstyle{definition} 

\newtheorem {thm}{Theorem}[section]
\newtheorem{cor}[thm]{Corollary}
\newtheorem{lem}[thm]{Lemma}
\newtheorem{prop}[thm]{Proposition}

\theoremstyle{definition}
\newtheorem{defn}{Definition}[section]

\theoremstyle{remark}
\newtheorem{rem}{Remark}[section]

\newcommand{\be}{\begin{equation}}
\newcommand{\ee}{\end{equation}}
\newcommand{\bea}{\begin{eqnarray}}
\newcommand{\eea}{\end{eqnarray}}
\newcommand{\ben}{\begin{eqnarray*}}
	\newcommand{\een}{\end{eqnarray*}}
\newcommand{\bt}{\begin{split}}
	\newcommand{\et}{\end{split}}
\newcommand{\bet}{\begin{equation}}

\begin{document}
\title{Invariance of plurigenera for K\"ahler families with nef canonical bundles}

\author[F. Deng]{Fusheng Deng}
\address{Fusheng Deng: \ School of Mathematical Sciences, University of Chinese Academy of Sciences\\ Beijing 100049, P. R. China}
\email{fshdeng@ucas.ac.cn}

\author[J. Ning]{Jiafu Ning}
\address{Jiafu Ning: \ Department of Mathematics, Central South University, Changsha, Hunan 410083, P. R. China.}
\email{jfning@csu.edu.cn}
\author[Z. Wang]{Zhiwei Wang}
\address{Zhiwei Wang: Laboratory of Mathematics and Complex Systems (Ministry of Education)\\ School of Mathematical Sciences\\ Beijing Normal University\\ Beijing 100875\\ P. R. China}
\email{zhiwei@bnu.edu.cn}
\author[X. Zhou]{Xiangyu Zhou}
\address{Xiangyu Zhou: Institute of Mathematics\\Academy of Mathematics and Systems Sciences\\and Hua Loo-Keng Key
	Laboratory of Mathematics\\Chinese Academy of
	Sciences\\Beijing\\100190\\P. R. China}
\address{School of
	Mathematical Sciences, University of Chinese Academy of Sciences,
	Beijing 100049, China}
\email{xyzhou@math.ac.cn}

\begin{abstract}

Let $p:X\rightarrow\Delta$ be a proper holomorphic submersion from a
K\"ahler manifold $X$. We prove that, if $K_{X_t}$ is nef for every
$t\in\Delta$, then $\dim_{\mathbb C}H^0(X_t,mK_{X_t})$ is independent of
$t$ for every integer $m\geq1$. This gives a partial answer to Siu's
conjecture on the invariance of plurigenera for K\"ahler families.
 
\end{abstract}

\thanks{
This research is supported by National Key R\&D Program of China (No. 2021YFA1002600). The authors are partially supported respectively by NSFC grants (11871451,  11801572, 12071035, 11688101).
The first author is  partially supported by the University of Chinese Academy of Sciences.
The third author is partially supported by Beijing Natural Science Foundation (1202012, Z190003).}

\maketitle
\tableofcontents
\section{Introduction}
The main purpose of this paper is to study Siu's conjecture on the
deformation invariance of plurigenera. We use the total-space K\"ahler
formulation stated in \cite[Conjecture 1.3]{CP20}, corresponding to
Siu's formulation for holomorphic families of compact K\"ahler manifolds
\cite[Conjecture 2.1]{Siu022}.
\begin{conj}\label{conj: Siu}
Let $p:X\rightarrow\Delta$ be a proper holomorphic submersion from a
K\"ahler manifold $X$. Then, for every integer $m\geq1$, the number
\[
 P_m(X_t):=\dim_{\mathbb C}H^0(X_t,mK_{X_t})
\]
is independent of $t\in\Delta$.
\end{conj}

The projective counterpart plays an important role in birational geometry.
Siu proved invariance for smooth projective families of manifolds of general
type in \cite{Siu98}, and Kawamata subsequently gave an algebraic proof in
\cite{Kaw99}. Siu proved the general smooth projective case in
\cite{Siu021}; P\u aun later gave a one-tower proof in \cite{Pau07}.
Takayama proved an extension to algebraic families over a smooth curve whose
fibers have canonical singularities \cite[Theorem 1.1]{Tak07}.

The study of the K\"ahler case goes back to Levine \cite{L83,L85}. The
1983 result assumes the existence of a nonzero pluricanonical section with
smooth divisor, while the 1985 result allows the associated cyclic cover to
have the mild singularities specified there.

Cao--P\u aun proved the following further extension criterion
\cite[Corollary 1.4]{CP20}. Given a nonzero section
$s\in H^0(X_0,mK_{X_0})$, let $\Sigma_s=\operatorname{div}(s)$ and set
\[
 \mathfrak I_s:=\lim_{\varepsilon\downarrow0}
 \mathcal I\!\left((1-\varepsilon)\frac{m-1}{m}\Sigma_s\right),
\]
where the limit denotes the stable value for $0<\varepsilon\ll1$. If
the zero set of $\mathfrak I_s$ is discrete, then $s$ extends
holomorphically to the total space. This contains Levine's criterion as
a special case.

The main result of the present paper is the following.
\begin{thm}[=Theorem \ref{thm: inv-Kah}]
Let $p:X\rightarrow\Delta$ be a proper holomorphic submersion from a
K\"ahler manifold $X$, and assume that $K_{X_t}$ is nef for every
$t\in\Delta$. Then, for every integer $m\geq1$, every $t_0\in\Delta$, and
every $s\in H^0(X_{t_0},mK_{X_{t_0}})$, there exists
$\widetilde s\in H^0(X,mK_X)$ whose restriction to $X_{t_0}$, under the
identification induced by the standard coordinate on $\Delta$, is $s$.
Consequently $P_m(X_t)$ is independent of $t\in\Delta$.
\end{thm}

The proof of above theorem relies on the following important  ingredients:
\begin{itemize}
\item P\u aun's total-space positivity for fiberwise twisted
K\"ahler--Einstein metrics \cite{Pa17}.
\item The uniform, volume-independent estimates for degenerate complex
Monge--Amp\`ere equations in families due to Di Nezza--Guedj--Guenancia
\cite{DNGG}.
\item The Ohsawa--Takegoshi type extension theorem for weakly pseudoconvex
K\"ahler manifolds and singular metrics, cf. \cite{ZZ17,ZZ20}, where Guan-Zhou's solution to Demailly's strong openness conjecture for multiplier ideal sheaves plays a key role \cite{GZ15}.
\end{itemize}
    
Let us explain the main strategy. For $\eta>0$, P\u aun's construction
produces a smooth potential $\phi_\eta$ such that
\[
 \theta+\eta\omega+\sqrt{-1}\partial\bar\partial\phi_\eta>0
\]
on the total space, where $\theta$ represents $c_1(K_{X/\Delta})$.
After subtracting the fiberwise supremum, the Di Nezza-Guedj-Guenancia estimate compares
$\phi_\eta$ with the extremal envelope in
$c_1(K_{X_t})+\eta[\omega_t]$ by a constant independent of $t$, $\eta$,
and the possibly collapsing class volume. Integrating the
Monge--Amp\`ere equation shows that the fiberwise suprema differ from one
another by a uniformly bounded amount. We can therefore subtract one
total-space constant and pass to a quasi-psh limit $F$. Its restriction to
$X_0$ has minimal singularities in $c_1(K_{X_0})$. For
$s\in H^0(X_0,mK_{X_0})$, the metric induced by $(m-1)F$ then satisfies
\[
 |s|^2e^{-(m-1)F}\leq C|s|^{2/m},
\]
up to a smooth factor. The right-hand side is integrable, so the
Ohsawa--Takegoshi theorem applies.

\subsection*{Acknowledgement}
We wish to express our deepest gratitude to Professor Mihai P\u aun for his beautiful lectures in  Zhou's seminar in Chinese Academy of Sciences, where he announced his proof of the Siu conjecture in the nef canonical bundle case. We had independently obtained our proof and sent to him our manuscript written before his lectures. Although we proposed a joint paper, Professor P\u{a}un, with kind generosity, encouraged us to upload it as soon as possible under our own names. We sincerely thank him for his generosity, encouragement, and valuable discussions for further work.

\section{Preparations}
In this section, we recall the definition and collect some known results  and prepare lemmas and propositions, related to   quasi-psh functions.
\
\begin{defn}Let $X$ be a complex manifold and $\varphi:X\rightarrow [-\infty,+\infty)$ be an upper semicontinuous function on $X$.  We say $\varphi$ is quasi-plurisubharmonic (quasi-psh for short) if locally, it can be written as a sum of smooth function and psh function. We say a quasi-psh function $\varphi$ is with analytic singularities \cite{Dem12}, if locally it can be  written as 
	\begin{align*}
		\varphi=c\log\!\left(\sum_j|f_j|^2\right)+g,
		\end{align*}
	where $c\in \mathbb R_+$ and $g$ is a smooth function. If  here $g$ is only continuous, then we say $\varphi$ is with weak analytic singularities \cite{CT15}.
\end{defn}
\begin{defn}

Let $X$ be a complex manifold and  $\varphi$ be a quasi-psh function on an open subset $\Omega\subset X$. The multiplier ideal sheaf $\mathcal I(\varphi)\subset \mathscr O_\Omega$ associated to $\varphi$ is the set of germs of holomorphic functions $f\in \mathscr O_{\Omega,x}$ such that $|f|^2e^{-\varphi}$ is integrable with respect to the Lebesgue measure in some local coordinates near $x$.
\end{defn}
It is trivial to see that if $\varphi_1<\varphi_2$ on $\Omega$, then it holds that $\mathcal I(\varphi_1)\subset \mathcal I(\varphi_2)$. 
\
\subsection{Direct image formula}
\begin{prop}[cf. {\cite[Proposition 5.8]{Dem12}}]\label{prop: direct image}
Let $\pi:\tilde X\rightarrow X$ be a modification of non singular complex manifolds (i.e. a proper generically $1:1$ holomorphic map), and let $\varphi$ be a quasi-psh function on $X$. Then 
$$\pi_*(\mathscr O(K_{\tilde X})\otimes \mathcal I(\varphi\circ \pi))=\mathscr O(K_X)\otimes \mathcal I(\varphi),$$
and thus
$$\mathcal I(\varphi)=\pi_*(\mathscr O(K_{\tilde X/X})\otimes \mathcal I(\varphi\circ \pi)),$$
where $K_{\tilde X/X}:=K_{\tilde X}-\pi^*K_X$.
\end{prop}

\begin{rem}\label{rem: compare mis}
Suppose that $\psi$ is a quasi-psh function on $\tilde X$, such that $\mathcal I(\varphi\circ \pi)\subset\mathcal I(\psi) $,  from Proposition \ref{prop: direct image}, we have that 
$$\mathscr O(K_X)\otimes \mathcal I(\varphi)=\pi_*(\mathscr O(K_{\tilde X})\otimes \mathcal I(\varphi\circ \pi))\subset\pi_*(\mathscr O(K_{\tilde X})\otimes \mathcal I(\psi)). $$
Thus, 
$$\mathcal I(\varphi)\subset \pi_*(\mathscr O(K_{\tilde X/X})\otimes \mathcal I(\psi)).$$
\end{rem}
\
\subsection{Strong openness} Let $\varphi$ be a quasi-psh function on a complex manifold $X$.
Associated to $\varphi$, there is another ideal sheaf 
$$\mathcal I_+(\varphi):=\cup_{\varepsilon>0}\mathcal I((1+\varepsilon)\varphi).$$

\begin{thm}[\cite{GZ15}]\label{thm: soc}
$$\mathcal I_+(\varphi)=\mathcal I(\varphi).$$
\end{thm}
\begin{lem}\label{lemma:soc}
Let $\varphi$ and $\psi$ be quasi-plurisubharmonic functions on a
complex manifold $X$, neither identically equal to $-\infty$ on any
connected component. Then, for every $x\in X$, there exist an open
neighborhood $U_x$ of $x$ and a number $\varepsilon_x>0$ such that
\[
 \mathcal I(\varphi)|_{U_x}
 =
 \mathcal I(\varphi+\varepsilon\psi)|_{U_x},
 \qquad 0<\varepsilon\leq\varepsilon_x .
\]
Consequently, for every compact subset $K\Subset X$, there exist an
open neighborhood $U\supset K$ and a number $\varepsilon_K>0$ such that
\[
 \mathcal I(\varphi)|_U
 =
 \mathcal I(\varphi+\varepsilon\psi)|_U,
 \qquad 0<\varepsilon\leq\varepsilon_K .
\]
In particular, when $X$ is compact, one may choose a single
$\varepsilon_X>0$ for all points of $X$.
\end{lem}

\begin{proof}
The assertion is local. After adding constants, which does not change
multiplier ideal sheaves, we may assume that $\varphi\leq0$ and
$\psi\leq0$ on a relatively compact coordinate neighborhood of a fixed
point $x$.

By the coherence of $\mathcal I(\varphi)$, after shrinking the
neighborhood we may choose holomorphic generators $f_1,\ldots,f_N$ of
$\mathcal I(\varphi)$. The strong openness theorem, applied to these
finitely many generators, gives a number $p>1$ and a smaller
neighborhood $U_x$ such that
\[
 \int_{U_x}|f_j|^2e^{-p\varphi}\,d\lambda<+\infty,
 \qquad 1\leq j\leq N.
\]
Put $q=p/(p-1)$. Write locally $\psi=\rho+g$, where $\rho$ is psh and
$g$ is smooth. Since $\nu(\rho,x)<+\infty$, choose $b>0$ such that
$\nu((b/2)\rho,x)<1$. By \cite[Lemma 5.6(a)]{Dem12},
$e^{-b\rho}$ is integrable near $x$; the smooth term $g$ therefore
gives, after shrinking $U_x$ once more,
\[
 \int_{U_x}e^{-b\psi}\,d\lambda<+\infty.
\]
Choose $\varepsilon_x>0$ so that $q\varepsilon_x<b$. H\"older's
inequality then gives, for $0<\varepsilon\leq\varepsilon_x$,
\[
\begin{aligned}
 \int_{U_x}|f_j|^2e^{-\varphi-\varepsilon\psi}\,d\lambda
 &=
 \int_{U_x}
 \bigl(|f_j|^2e^{-p\varphi}\bigr)^{1/p}
 \bigl(|f_j|^2e^{-q\varepsilon\psi}\bigr)^{1/q}
 \,d\lambda                                                   \\
 &\leq
 \left(\int_{U_x}|f_j|^2e^{-p\varphi}\,d\lambda\right)^{1/p}
 \left(\int_{U_x}|f_j|^2e^{-q\varepsilon\psi}\,d\lambda
 \right)^{1/q}<+\infty .
\end{aligned}
\]
Here the second factor is finite because the $f_j$ are bounded after a
further shrinking of $U_x$. Hence
\[
 \mathcal I(\varphi)|_{U_x}
 \subset
 \mathcal I(\varphi+\varepsilon\psi)|_{U_x}.
\]
Since $\psi\leq0$, monotonicity gives the opposite inclusion. This
proves the local assertion. For $K\Subset X$, take a finite covering by
neighborhoods of the above type and take the minimum of the corresponding
positive constants.
\end{proof}

Let $\mu:Y\to X$ be a modification and let $T$ be a closed almost
positive $(1,1)$-current on $X$. Locally write
\[
 T=\theta+\sqrt{-1}\partial\bar\partial\varphi,
\]
where $\theta$ is smooth and $\varphi$ is quasi-psh. Since no connected
component of $Y$ is mapped into the polar set of a local potential,
$\varphi\circ\mu$ is locally integrable, and we define
\[
 \mu^*T:=\mu^*\theta+
 \sqrt{-1}\partial\bar\partial(\varphi\circ\mu).
\]
These local definitions glue and are independent of the chosen
potentials.

\begin{lem}[{\cite[Satz3, Satz 4]{GR56}}]\label{lem: ex-psh}
	Let $X$ be a complex manifold, and $A\subset X$ be an analytic subset. Then 
	\begin{itemize}
 	\item[(1)] every plurisubharmonic function on $X\setminus A$ that is locally bounded from above near any point in $A$ extends uniquely to a plurisubharmonic function on $X$;
 	\item[(2)] if codim$_{\mathbb C}A\geq 2$ in $X$, every plurisubharmonic function on $X\setminus A$ extends  uniquely to a plurisubharmonic function on $X$.
		\end{itemize}
	\end{lem}
\begin{lem}\label{lem: push-ext}
Let $\pi:\widetilde X\to X$ be a proper modification between complex
manifolds. Assume there is a compact analytic subset $E\subset X$ with
$\operatorname{codim}_{\mathbb C}E\geq2$ such that, putting
$A:=\pi^{-1}(E)$, the restriction
\[
 \pi:\widetilde X\setminus A\longrightarrow X\setminus E
\]
is biholomorphic. Let $\alpha$ be a smooth closed real $(1,1)$-form on
$X$, and let $T$ be a closed positive $(1,1)$-current on $\widetilde X$.
Assume that there is a global quasi-psh function $\psi$ on
$\widetilde X$ such that
\[
 T=\pi^*\alpha+\sqrt{-1}\partial\bar\partial\psi.
\]
Then $\pi_*T|_{X\setminus E}$ extends uniquely to a closed positive
$(1,1)$-current $S$ on $X$. Moreover,
\[
 [S]=[\alpha]\quad\text{in Bott--Chern cohomology computed with currents},
 \qquad \pi^*S=T.
\]
We denote this extension by $S=\pi_*T$.
\end{lem}
\begin{proof}
On $X\setminus E$, define
\[
 \phi:=\psi\circ
 \bigl(\pi|_{\widetilde X\setminus A}\bigr)^{-1}.
\]
Then $\pi_*T=\alpha+\sqrt{-1}\partial\bar\partial\phi$ there.

Fix $x\in E$ and choose a relatively compact coordinate neighborhood
$U$ such that $\alpha|_U=\sqrt{-1}\partial\bar\partial u$ for a smooth
real-valued function $u$. Properness and the local upper boundedness of
quasi-psh functions on compact sets show that $\phi$ is locally bounded
from above near $E\cap U$. Thus $u+\phi$ is psh on $U\setminus E$ and,
by Lemma \ref{lem: ex-psh}, has a unique psh extension $v_U$ to $U$.
The functions $\phi_U:=v_U-u$ agree on overlaps and define a global
quasi-psh function $\widehat\phi$ on $X$. Hence
\[
 S:=\alpha+\sqrt{-1}\partial\bar\partial\widehat\phi
\]
is a closed positive extension and $[S]=[\alpha]$.

If $S'$ is another closed positive extension, put $R:=S-S'$. Then $R$
is a normal current of real degree $2$ supported on $E$. On the regular
stratum of $E$ we have
$\operatorname{codim}_{\mathbb R}E_{\mathrm{reg}}\geq4$; hence
\cite[Exercise 1.21]{Dem12} gives $R=0$ near $E_{\mathrm{reg}}$.
Thus $\operatorname{Supp}R\subset E_{\mathrm{sing}}$, and descending
induction on the dimension of the analytic support gives $R=0$.

Finally, $\widehat\phi\circ\pi=\psi$ on
$\widetilde X\setminus A$. Both sides are locally integrable quasi-psh
functions on $\widetilde X$ and $A$ has Lebesgue measure zero, so they
define the same distribution. Therefore
\[
 \pi^*S
 =\pi^*\alpha+
 \sqrt{-1}\partial\bar\partial(\widehat\phi\circ\pi)
 =T.
\]
\end{proof}
\
\

\subsection{Complex analytic desingularizations}
\

In this section, we recall Hironaka's fundamental work of desingularizations of complex analytic spaces.
\begin{thm}[cf. \cite{Hir77}, {\cite[Theorem 3, Theorem 4]{AHV18}}]\label{thm: resolution}
Given a coherent ideal sheaf $J\neq (0)$ on a complex manifold $M$, there exists a complex manifold $\widetilde M$ and a proper bimeromorphic map $\tilde \pi:\widetilde M\rightarrow M$ such that the pull-back $\tilde J$ of $J$ by $\tilde \pi$ is locally normal crossings everywhere in $\widetilde M$. The $\tilde \pi$ is locally obtained by a finite sequence of blow-ups with smooth centers. To be precise, let $K$ be any compact subset of $M$. Then there exists an open neighborhood $Z_0$ of $K$ in $M$ such that $\tilde \pi|_{Z_0}$ is decomposed into a finite sequence of blow-ups 
\begin{align*}
\pi_j:Z_{j+1}\rightarrow Z_j, ~~0\leq j\leq r,
\end{align*}
where the center $D_j$ of $\pi_j$ is such that 
\begin{itemize}
\item [(1)] $D_j$ is closed and nowhere dense in $Z_j$ and 
\item [(2)] $D_j$ is smooth by itself.
\item [(3)] Let $U_k,~k=1,2,$ be two open subsets of $Z_0$. Assume that we have an isomorphism $t: U_1\rightarrow U_2$ such that $J|_{U_1}$ is the pull-back of $J|_{U_2}$. Then for every $j\geq 0$, we have an isomorphism $t(j):\pi(j)^{-1}(U_1)\rightarrow \pi(j)^{-1}(U_2)$ which makes a commutative diagram:
\[
\xymatrix{
	\pi(j)^{-1}(U_1) \ar[d]_{\pi(j)}  \ar[r]^{t(j)} & \pi(j)^{-1}(U_2) \ar[d]^{\pi(j)} \\
     U_1 \ar[r]^t  & U_2 }
\]
where $\pi(j)$ denotes the composition of $\pi_j$ for all $i\leq j$.

\end{itemize}
\end{thm}

\begin{rem}
Here complex manifold $M$ is assumed to be connected and countable at infinity, i.e., it is a union of countably many compact subsets. Let $\dim_{\mathbb C}=\dim_{\mathbb C}\widetilde M=n$. The term local normal crossings means that  for every point $\widetilde \xi\in \widetilde M$ we have a local coordinate system $x=(x_1,\cdots, x_n)$ of $\widetilde M$ at $\widetilde \xi$ such that $\widetilde J$ is generated by a monomial in $x$ within a neighborhood of $\widetilde \xi$. If $n=1$, then $\tilde \pi$ must be an isomorphism and $J$ itself is local normal crossings from the beginning.
\end{rem}

\subsection{The Ohsawa-Takegoshi type extension theorem}
\begin{thm}[cf. {\cite[Theorem 1.1 and Remark 1.3]{ZZ17}}]
\label{thm: ext}
Let $(X,\omega)$ be a weakly pseudoconvex K\"ahler manifold of
dimension $n$, and let $L$ be a holomorphic line bundle over $X$ with a
singular Hermitian metric $h$. Let $s:X\rightarrow\mathbb C^r$ be a
holomorphic map, $1\leq r\leq n$, such that $0$ is a regular value of
$s$. Assume that $\sqrt{-1}\Theta_{L,h}\geq0$ and that
$|s(x)|\leq M$ on $X$. Put $Y:=s^{-1}(0)$ and endow
$\mathbb C^r$ with its standard flat Hermitian metric. The norms of
$\bigwedge^r(ds)$, $f$, and $F$ below are induced by this metric,
$\omega$, and $h$, and
\[
 dV_{X,\omega}:=\frac{\omega^n}{n!},
 \qquad
 dV_{Y,\omega}:=\frac{(\omega|_Y)^{n-r}}{(n-r)!}.
\]
Then every
\[
 f\in H^0\!\left(Y,(K_X\otimes L)|_Y\right)
\]
satisfying
\[
 \int_Y\frac{|f|_L^2}{|\bigwedge^r(ds)|^2}\,
 dV_{Y,\omega}<+\infty
\]
admits an extension $F\in H^0(X,K_X\otimes L)$ with $F|_Y=f$ and
\[
 \int_X|F|_L^2\,dV_{X,\omega}
 \leq C_{r,M}
 \int_Y\frac{|f|_L^2}{|\bigwedge^r(ds)|^2}\,
 dV_{Y,\omega},
\]
where $C_{r,M}$ depends only on $r$ and $M$.
\end{thm}

\begin{rem}
This is a direct specialization of
\cite[Theorem 1.1 and Remark 1.3]{ZZ17}. Let
$E:=X\times\mathbb C^r$ carry its standard flat metric and put
\[
 \lambda:=\frac{1}{e\max\{1,M\}},
 \qquad s_\lambda:=\lambda s.
\]
Then $|s_\lambda|\leq e^{-1}$. Apply the cited theorem with
$\psi=0$, $\alpha=1$, and $R(\tau)=e^{-\tau}$. The local weights of
$h$ are plurisubharmonic, hence locally bounded above, and
$\Theta_E=0$. Moreover, on $X\setminus Y$,
\[
 \sqrt{-1}\Theta_{L,h}
 +r\sqrt{-1}\partial\bar\partial\log|s_\lambda|^2\geq0,
 \qquad
 r\log|s_\lambda|^2\leq-2r,
\]
so all three curvature-side hypotheses of
\cite[Theorem 1.1]{ZZ17} hold. Since $C_R=1$ and
$e^\tau R(\tau)=1$, its estimate becomes
\[
 \int_X|F|_L^2\,dV_{X,\omega}
 \leq\frac{(2\pi)^r}{r!}
 \int_Y\frac{|f|_L^2}
 {|\bigwedge^r(ds_\lambda)|^2}\,dV_{Y,\omega}.
\]
As $\bigwedge^r(ds_\lambda)=\lambda^r\bigwedge^r(ds)$, one may take
\[
 C_{r,M}=\frac{(2\pi)^r}{r!}
 \bigl(e\max\{1,M\}\bigr)^{2r}.
\]
A proper-family formulation with additional multiplier-ideal conclusions
appears in \cite[Theorem 1.3]{ZZ20}.
\end{rem}


%
%
%
%
%
%
%

\section{Extension of quasi-psh functions  from invariance of plurigenera}
In this section, we prove an "extension" theorem for 
quasi-psh functions by using Siu's invariance of plurigenera and Demailly's approximation theorems.

\begin{thm}[\cite{Siu98}]\label{thm: siu-inv}Let $p:X\rightarrow \Delta$ be a smooth proper holomorphic projective submersion, and denote by  $X_t:=p^{-1}(t)$. Suppose for every $t\in \Delta$, $X_t$ is of general type, then for any $m\in \mathbb Z^+$, any $\sigma\in H^0(X_t,mK_{X_t})$ can be extended to $\tilde{\sigma}\in H^0(X,mK_X)$.
\end{thm}
As a direct consequence of Siu's invariance of plurigenera, the singular metric constructed from pluricanonical sections on the fiber can be extended to the ambient space to be a singular metric of pluricanonical bundles. Combined with the following result of Demailly (Theorem \ref{thm: dem-supercan}), we can get that any singular metric of the pluricanonical bundles on the fibers can be "extended" to the ambient space, by keeping the size of the associated multiplier ideal sheaf, i.e. Theorem \ref{thm: dnwz-1}.

\begin{thm}[{\cite[Proposition 19.8]{Dem12}}]\label{thm: dem-supercan} Let $X$ be a projective manifold  $\omega$ be a Hermitian metric on $X$. Let $h_{K_X} $ be the induced metric on $K_X$ by $\omega$, and $\alpha=\Theta_{h_{K_X}}$.  Assume that $K_X$ is big and fix a singular Hermitian metric $e^{-\varphi}h_{K_X+L}$ of curvature $\alpha+dd^c\varphi\geq 0$.  Then $\varphi$ is equal to a regularized limit 
\begin{align*}
\varphi=\left( \limsup_{m\rightarrow \infty}\frac{1}{m}  \log|\sigma_m|^2_{h^m_{K_X}}       \right)^*
\end{align*}
for a suitable sequence $\sigma_m\in H^0(X,mK_X)$ with $\int_X(\sigma_m\wedge\bar\sigma_m)^{1/m}\leq 1$.
\end{thm}
\begin{rem} It is quite interesting that if there is similar  approximation  as above   for $K_X$ pseudo-effective.
	\end{rem}

Let $p:X\rightarrow \Delta$ be a smooth proper holomorphic projective submersion, and denote by  $X_t:=p^{-1}(t)$. Suppose for every $t\in \Delta$, $X_t$ is of general type. Let $\omega$ be a Hermitian metric on $X$, and $h_{K_{X/\Delta}}$ be the induced metric on $K_{X/\Delta}\simeq K_X$ by $\omega$. Let $\alpha$ be the corresponding curvature. Let $\varphi$ be an $\alpha|_{X_t}$-psh function on $X_t$. Since $X_t$ is of general type, then from Theorem \ref{thm: dem-supercan}, there  is a suitable chosen sequence $\sigma_m\in H^0(X_t, mK_{X_t})$, such that 
\begin{align*}
\varphi=\left( \limsup_{m\rightarrow \infty}\frac{1}{m}  \log|\sigma_m|^2_{h^m_{K_{X_t}}}       \right)^*
\end{align*}
with $\int_X(\sigma_m\wedge\bar\sigma_m)^{1/m}\leq 1$.
Now from Siu's invariance of plurigenera, i.e. Theorem \ref{thm: siu-inv}, we have $\widetilde\sigma_m\in H^0(X,mK_{X})$ such that $\widetilde\sigma_m|_{X_t}=\sigma_m\wedge dt$. Put
\begin{align*}
\Phi=\left(\limsup_{m\rightarrow \infty}\frac{1}{m}  \log|\widetilde\sigma_m|^2_{h^m_{K_X}} \right)^*
\end{align*}
Then it is easy to see that $\alpha+dd^c\Phi\geq 0$ in the sense of current, and $\Phi|_{X_t}\geq \varphi$, thus $\mathcal I(\varphi)\subset \mathcal I(\Phi|_{X_t})$. 
We thus proved the following 
\begin{thm}\label{thm: dnwz-1}
Let $p:X\rightarrow \Delta$ be a smooth proper holomorphic projective submersion, and denote by  $X_t:=p^{-1}(t)$. 
Suppose for every $t\in \Delta$, $X_t$ is of general type. Let $\omega$ be a Hermitian metric on $X$, and $h_{K_{X/\Delta}}$ be the induced metric on $K_{X/\Delta}\simeq K_X$ by $\omega$. 
Let $\alpha$ be the corresponding curvature. 
Let $\varphi$ be an $\alpha|_{X_t}$-psh function on $X_t$. 
Then there is an $\alpha$-psh function $\Phi$ on $X$, such that $\Phi|_{X_t}\geq \varphi$ on $X_t$ and thus $\mathcal I(\varphi)\subset \mathcal I(\Phi|_{X_t})$.
\end{thm}
\begin{rem}\label{rem:dnwz-1}
Let $\varphi$ be an $m\alpha|_{X_t}$-psh function on $X_t$, that is $m\alpha|_{X_t}+dd^c\varphi\geq 0$. Then $\alpha|_{X_t}+dd^c\frac{\varphi}{m}\geq 0$. 
From Theorem \ref{thm: dnwz-1}, we can find an $\alpha$-psh function $\Phi$ on $X$, such that $\Phi\geq \frac{\varphi}{m}$ on $X_t$. 
So $m\Phi\geq \varphi$ on $X_t$ and $\mathcal I(\varphi)\subset \mathcal I(m\Phi|_{X_t})$, $m\alpha+dd^cm\Phi\geq 0$.
\end{rem}

\section{Extension of quasi-psh functions with control of multiplier ideal sheaves}
Let $p:X\rightarrow \Delta$ be a proper holomorphic submersion, with $X$ a K\"ahler manifold. Let $\omega$ be a K\"ahler metric on $X$.  Assume that $K_{X_t}$ is nef for any $t\in \Delta$. The following useful  lemma is due to P\u aun.
\begin{defn}[\cite{Dem90}] Let $[\alpha]\in H_{BC}^{1,1}(X,\mathbb R)$ be a real $(1,1)$-class with smooth representative $\alpha$.  We say $[\alpha]$ is nef if for every $\varepsilon>0$ there exists a function $f_\varepsilon\in \mathcal C^\infty(X)$ such that 
$$\alpha+i\partial\bar\partial f_\varepsilon\geq -\varepsilon\omega.$$
\end{defn}
\begin{rem}
On a compact Hermitian manifold, the preceding metric definition is the
one introduced in \cite{Dem90}; for a line bundle on a projective
manifold it agrees with numerical effectiveness. The total space used
here is non-compact. We therefore keep the K\"ahler form $\omega$ fixed
and use the terminology only for the relative canonical bundle, for
which the required metric estimate is supplied by the following theorem.
\end{rem}

\begin{thm}[{\cite[Corollary 1.3]{Pa17}}]\label{lem: paun}
Assume that $K_{X_t}$ is nef for every $t\in\Delta$. Then
$K_{X/\Delta}$ is nef in the metric sense. If $t$ denotes the standard
coordinate on $\Delta$, the nowhere-vanishing form $dt$ trivializes
$K_\Delta$ and yields
\[
 K_X\simeq K_{X/\Delta}\otimes p^*K_\Delta\simeq K_{X/\Delta}.
\]
Thus the same assertion can be expressed for $K_X$ after this
trivialization.
\end{thm}


\
\

We now give the replacement needed for the application to pluricanonical
forms. Its proof uses the special Monge--Amp\`ere representatives of the
relative canonical class; it does not hold for an arbitrary nef class.

\begin{thm}\label{thm: family mid}
Let $p:X\rightarrow\Delta$ be a proper holomorphic submersion, with $X$
K\"ahler, and assume that $K_{X_t}$ is nef for every $t\in\Delta$.
Let $k\geq1$, let $\alpha$ be a smooth real $d$-closed $(1,1)$-form
representing the Bott--Chern class $c_1(kK_{X/\Delta})$, and let
$\varphi\in\operatorname{PSH}(X_0,\alpha|_{X_0})$.
Then, for every $0<r<1$, there exists
\[
 \Phi\in\operatorname{PSH}
 \bigl(p^{-1}(\Delta_{1-r}),\alpha\bigr)
\]
such that
\[
 \varphi\leq \Phi|_{X_0}+O(1),
 \qquad
 \mathcal I(\varphi)\subset\mathcal I(\Phi|_{X_0}).
\]
In particular, the conclusion holds without assuming that $\varphi$ has
analytic singularities.
\end{thm}

\begin{proof}
We first reduce to the case where $X$ is connected.  Indeed, every
connected component of $X$ maps onto $\Delta$: its image is open because
$p$ is a submersion and closed because $p$ is proper.  There are only
finitely many such components, since the compact manifold $X_0$ has
finitely many connected components.  We can therefore carry out the
construction below on each component and combine the resulting functions,
which are defined on mutually disjoint open-and-closed subsets of $X$.
Assume henceforth that $X$ is connected. Since $p$ is a proper
submersion, Ehresmann's theorem shows that it is a differentiable
locally trivial fibration. In particular, the connected components of
the fibers form a finite covering of $\Delta$. Equivalently, in the
Stein factorization
\[
 X\longrightarrow S\longrightarrow\Delta,
\]
the finite map $S\to\Delta$ is unramified. Since $X$, hence $S$, is
connected and $\Delta$ is simply connected, this covering has degree
one. Thus every fiber $X_t$ is connected and, being a complex
manifold, is irreducible.

Fix $R:=1-r$ and choose $R<R'<R''<1$. Let $\omega$ be a K\"ahler form on
$X$, let $h^\omega$ be the smooth metric on $K_{X/\Delta}$ induced by
$\omega$, and set
\[
 \theta:=\Theta_{h^\omega}(K_{X/\Delta}).
\]
Write $n:=\dim_{\mathbb C}X_t$, which is independent of $t$.
Since the smooth real $d$-closed forms $\alpha$ and $k\theta$ represent
the same Bott--Chern class, there exists a smooth real-valued function
$q$ on $X$ such that
\begin{equation}\label{eq:alpha-k-theta}
 \alpha=k\theta+\sqrt{-1}\partial\bar\partial q.
\end{equation}

For $0<\eta\leq1$, the class
$[\theta_t+\eta\omega_t]=c_1(K_{X_t})+\eta[\omega_t]$ is K\"ahler.
P\u aun's fiberwise twisted K\"ahler--Einstein construction
\cite[Theorem 3.2 and equations (8), (10), (11)]{Pa17} gives smooth
functions $\phi_{\eta,t}$ satisfying
\begin{equation}\label{eq:twisted-ke}
 (\theta_t+\eta\omega_t+
   \sqrt{-1}\partial\bar\partial\phi_{\eta,t})^n
 =e^{\phi_{\eta,t}}\omega_t^n.
\end{equation}
The construction and the curvature computation are local on the base,
in particular, \cite[Section 4.1]{Pa17} applies them to a K\"ahler
family over a disk. The uniqueness in the fiberwise equation and smooth
dependence for the corresponding elliptic equation show that the
functions $\phi_{\eta,t}$ define a smooth function $\phi_\eta$ on $X$.

Let $c(\rho_\eta)$ denote the geodesic curvature of the form below, and
let $v$ be the horizontal lift of $\partial/\partial t$. P\u aun's
calculation \cite[Section 3.2, equation (35)]{Pa17} gives
\[
 \Box_t c(\rho_\eta)
 =-c(\rho_\eta)+|\bar\partial v|^2
   +\eta\omega(v,\bar v).
\]
If $x_{\min}\in X_t$ is a minimum point of
$c(\rho_\eta)|_{X_t}$, then, with P\u aun's positive-spectrum
convention, $\Box_t c(\rho_\eta)(x_{\min})\leq0$. Hence
\[
 c(\rho_\eta)(x_{\min})
 \geq |\bar\partial v|^2(x_{\min})
      +\eta\omega(v,\bar v)(x_{\min})>0.
\]
The last inequality is strict because $dp(v)=\partial/\partial t$, so
$v\neq0$, and $\omega$ is K\"ahler. Together with positivity in the
vertical directions, this proves
\begin{equation}\label{eq:paun-total-positive}
 \rho_\eta:=\theta+\eta\omega+
 \sqrt{-1}\partial\bar\partial\phi_\eta>0.
\end{equation}
Notice that no quantitative lower bound from
\cite[equation (36)]{Pa17} is used; in particular, no diameter-dependent
constant enters the argument.

For $|t|\leq R'$, put
\[
 \mathcal V_\eta:=\int_{X_t}(\theta_t+\eta\omega_t)^n,
 \qquad
 W:=\int_{X_t}\omega_t^n,
 \qquad
 d\nu_t:=W^{-1}\omega_t^n.
\]
Both $\mathcal V_\eta$ and $W$ are independent of $t$. By the same
fiber-integration argument as in \cite[Lemma 2.2]{DNGG}, integration
along the fibers commutes with $d$, and hence
\[
 d\,p_*\bigl((\theta+\eta\omega)^n\bigr)=0,
 \qquad d\,p_*(\omega^n)=0.
\]
These pushforwards are zero-forms on the connected disk, so they are
constant.  Notice that $\mathcal V_\eta$ may nevertheless depend on
$\eta$ and tend to zero as $\eta\downarrow0$.  Set
\[
 c_{\eta,t}:=\sup_{X_t}\phi_{\eta,t},\qquad
 u_{\eta,t}:=\phi_{\eta,t}-c_{\eta,t},\qquad
 J_{\eta,t}:=\int_{X_t}e^{u_{\eta,t}}d\nu_t.
\]

We claim that there are constants $a,A,M>0$, depending at most on
$R',R'',\omega,\theta$ and on the connected component under
consideration, but independent of $t$ and $\eta$, such that
\begin{equation}\label{eq:J-uniform}
 A^{-1/a}\leq J_{\eta,t}\leq1
\end{equation}
and
\begin{equation}\label{eq:u-envelope}
 V_{\eta,t}-M\leq u_{\eta,t}\leq V_{\eta,t},
\end{equation}
where
\[
 V_{\eta,t}:=
 \sup\{v\in\operatorname{PSH}(X_t,\theta_t+\eta\omega_t):v\leq0\}^{*}.
\]
Indeed, on the compact set $p^{-1}(\overline\Delta_{R''})$ one can choose
$C_0>0$ such
that
\[
 C_0\omega_t-(\theta_t+\eta\omega_t)\geq0
\]
for all $|t|\leq R''$ and $0<\eta\leq1$. Hence
\[
 \operatorname{PSH}(X_t,\theta_t+\eta\omega_t)
 \subset\operatorname{PSH}(X_t,C_0\omega_t).
\]
By the reduction at the beginning of the proof, all fibers are
irreducible. Moreover, a holomorphic submersion is locally trivial in the
sense of \cite[Assumption 3.2(2)]{DNGG}. For the inverse image of each
disk used below, the base and all fibers are connected; hence its total
space is connected and, being a complex manifold, irreducible. Thus the
restricted family satisfies Setting 2.1 of \cite{DNGG}. The fixed form
$\Theta:=C_0\omega$ satisfies
\[
 -C_0\omega\leq\Theta\leq C_0\omega,
\]
and each class $[\Theta|_{X_t}]=C_0[\omega_t]$ is K\"ahler, hence
pseudo-effective. Cover $\overline\Delta_{R'}$ by finitely many
coordinate disks whose concentric double disks are relatively compact in
$\Delta_{R''}$.  After translation and rescaling on each double disk, the
uniform Skoda estimate \cite[Theorem 3.4]{DNGG}, applied to the fixed form
$C_0\omega$, gives constants $a_j,C_j>0$ on its inner disk.  Set
\[
 a:=\min_j a_j>0,
 \qquad A:=\max\left\{1,\max_j\frac{C_j}{W}\right\}.
\]
Decreasing the exponent from $a_j$ to $a$ preserves the estimate because
$v-\sup_{X_t}v\leq0$.  Since $d\nu_t=W^{-1}\omega_t^n$, we obtain
\begin{equation}\label{eq:uniform-skoda}
 \int_{X_t}e^{-a(v-\sup_{X_t}v)}d\nu_t\leq A
\end{equation}
for all functions in the latter class. Applied to $u_{\eta,t}$, this and
Jensen's inequality give
\[
 \int_{X_t}u_{\eta,t}d\nu_t\geq-\frac{\log A}{a},
 \qquad
 J_{\eta,t}\geq
 \exp\!\left(\int_{X_t}u_{\eta,t}d\nu_t\right)
 \geq A^{-1/a}.
\]
The other inequality in \eqref{eq:J-uniform} follows from
$u_{\eta,t}\leq0$.

Integrating \eqref{eq:twisted-ke} and using Stokes' theorem gives the
exact identity
\begin{equation}\label{eq:fiber-normalization-inv}
 \mathcal V_\eta=e^{c_{\eta,t}}WJ_{\eta,t}.
\end{equation}
Consequently, after division by its total mass,
\eqref{eq:twisted-ke} becomes
\begin{equation}\label{eq:normalized-ke-inv}
 \mathcal V_\eta^{-1}
 (\theta_t+\eta\omega_t+
  \sqrt{-1}\partial\bar\partial u_{\eta,t})^n
 =\frac{e^{u_{\eta,t}}}{J_{\eta,t}}d\nu_t.
\end{equation}
Set
\[
 f_{\eta,t}:=\frac{e^{u_{\eta,t}}}{J_{\eta,t}}.
\]
By definition of $J_{\eta,t}$,
\[
 \int_{X_t}f_{\eta,t}\,d\nu_t=1.
\]
Moreover, \eqref{eq:J-uniform} and $u_{\eta,t}\leq0$ give
\[
 0\leq f_{\eta,t}\leq A^{1/a},
 \qquad
 \|f_{\eta,t}\|_{L^2(d\nu_t)}\leq A^{1/a}.
\]
We now check all the hypotheses of \cite[Theorem 1.9]{DNGG}. The class
$[\theta_t+\eta\omega_t]$ is K\"ahler and hence big, and its volume is
\[
 \operatorname{vol}([\theta_t+\eta\omega_t])=\mathcal V_\eta.
\]
The function $u_{\eta,t}$ is smooth, has
$\sup_{X_t}u_{\eta,t}=0$, and therefore has minimal singularities.
Equation \eqref{eq:uniform-skoda} is precisely hypothesis {\rm(H1)} for
the reference probability measure $d\nu_t$, while the preceding $L^2$
estimate is hypothesis {\rm(H2)} with $p=2$. Finally,
\eqref{eq:normalized-ke-inv} says that $u_{\eta,t}$ solves the
normalized Monge--Amp\`ere equation occurring in that theorem.
Consequently, \cite[Theorem 1.9]{DNGG} yields
\eqref{eq:u-envelope}. Its explicit constant depends only on $n$, the
fixed exponent $p=2$, and the uniform constants $a,A$ (and hence on the
bound $A^{1/a}$ in {\rm(H2)}), but not on $t$, $\eta$, or
$\mathcal V_\eta$. Thus no positive lower bound for
$\mathcal V_\eta$ is required.

Since $[\theta_0]=c_1(K_{X_0})$ is nef, it is pseudo-effective and
$\operatorname{PSH}(X_0,\theta_0)$ is non-empty. For completeness,
choose $\varepsilon_\ell\downarrow0$ and, by nefness applied with
$\varepsilon_\ell/2$, smooth functions $b_\ell$ such that
\[
 \theta_0+\varepsilon_\ell\omega_0+
 \sqrt{-1}\partial\bar\partial b_\ell
 \geq\frac{\varepsilon_\ell}{2}\omega_0>0,
 \qquad \sup_{X_0}b_\ell=0.
\]
The compactness theorem for quasi-psh functions gives, after extraction,
a non-identically $-\infty$ limit
$b\in\operatorname{PSH}(X_0,\theta_0)$. Thus the following extremal
envelope is not identically $-\infty$:
\[
 V_{\theta_0}:=
 \sup\{v\in\operatorname{PSH}(X_0,\theta_0):v\leq0\}^{*}.
\]

Every competitor defining $V_{\theta_0}$ is also a competitor defining
$V_{\eta,0}$, so \eqref{eq:u-envelope} gives
\begin{equation}\label{eq:central-envelope-inv}
 u_{\eta,0}\geq V_{\theta_0}-M.
\end{equation}

Solving \eqref{eq:fiber-normalization-inv} for $c_{\eta,t}$ and using
\eqref{eq:J-uniform}, we obtain
\begin{equation}\label{eq:fiber-sup-inv}
 \log\mathcal V_\eta-\log W
 \leq c_{\eta,t}
 \leq\log\mathcal V_\eta-\log W+\frac{\log A}{a}.
\end{equation}
Although $\log\mathcal V_\eta$ may tend to $-\infty$, it is common to all
fibers and cancels in their comparison.

We now use one total-space constant, rather than the fiberwise constants
$c_{\eta,t}$. Define
\[
 C_\eta:=\sup_{p^{-1}(\overline\Delta_{R'})}\phi_\eta
          =\max_{|t|\leq R'}c_{\eta,t},
 \qquad F_\eta:=\phi_\eta-C_\eta.
\]
Equations \eqref{eq:paun-total-positive},
\eqref{eq:central-envelope-inv}, and \eqref{eq:fiber-sup-inv} imply
\begin{align}
 F_\eta&\leq0
       &&\text{on }p^{-1}(\overline\Delta_{R'}),\label{eq:F-upper-inv}\\
 \theta+\sqrt{-1}\partial\bar\partial F_\eta
      &\geq-\eta\omega,\label{eq:F-curvature-inv}\\
 F_\eta|_{X_0}&\geq V_{\theta_0}-M-\frac{\log A}{a}.
       \label{eq:F-central-inv}
\end{align}

Choose $\eta_j\downarrow0$. Since $V_{\theta_0}\not\equiv-\infty$,
we can fix $x_*\in X_0$ with $V_{\theta_0}(x_*)>-\infty$. Equation
\eqref{eq:F-central-inv} gives
\[
 \inf_j F_{\eta_j}(x_*)>-\infty.
\]
On the relatively compact domain
\[
 p^{-1}(\Delta_R)\Subset p^{-1}(\Delta_{R'}),
\]
the functions $F_{\eta_j}$ have the common upper bound zero and satisfy
the uniform quasi-psh estimate
\[
 \sqrt{-1}\partial\bar\partial F_{\eta_j}
 \geq-\theta-\eta_j\omega.
\]
The compactness dichotomy for quasi-psh functions therefore applies.
The lower bound at $x_*$, together with the submean inequality in a
coordinate ball centered at $x_*$, excludes the alternative that the
sequence converges locally uniformly to $-\infty$. After passing to a
subsequence, it converges in $L^1_{\mathrm{loc}}$ on
$p^{-1}(\Delta_R)$, and stability of the current inequality gives
\[
 F:=\left(\limsup_{j\to\infty}F_{\eta_j}\right)^*
 \in\operatorname{PSH}(p^{-1}(\Delta_R),\theta).
\]
Since every $F_{\eta_j}\leq0$ on the larger domain
$p^{-1}(\Delta_{R'})$, we also have $F\leq0$ on
$p^{-1}(\Delta_R)$.

For every $x\in X_0$, upper semicontinuous regularization gives
\[
 \begin{aligned}
 F(x)
 &\geq\limsup_{j\to\infty}F_{\eta_j}(x)\\
 &\geq V_{\theta_0}(x)-M-\frac{\log A}{a}.
 \end{aligned}
\]
In particular, $F|_{X_0}\not\equiv-\infty$, so the restriction is a
well-defined $\theta_0$-psh function. Since it is also non-positive, the
definition of $V_{\theta_0}$ yields
\[
 V_{\theta_0}-M-\frac{\log A}{a}
 \leq F|_{X_0}\leq V_{\theta_0}.
\]
Thus $F|_{X_0}$ has minimal singularities in $c_1(K_{X_0})$.

Finally, define
\[
 \Phi:=kF-q
\]
Equation \eqref{eq:alpha-k-theta} gives
\[
 \alpha+\sqrt{-1}\partial\bar\partial\Phi
 =k\bigl(\theta+\sqrt{-1}\partial\bar\partial F\bigr)\geq0,
\]
so $\Phi$ is $\alpha$-psh. Moreover,
\[
 \psi:=\frac{\varphi+q|_{X_0}}{k}
 \in\operatorname{PSH}(X_0,\theta_0).
\]
Since $F|_{X_0}$ has minimal singularities, there is a constant $B$
such that $\psi\leq F|_{X_0}+B$. Therefore
\[
 \varphi\leq kF|_{X_0}-q|_{X_0}+kB
            =\Phi|_{X_0}+kB.
\]
With the multiplier-ideal convention used in this paper, this inequality
gives locally
\[
 e^{-\Phi|_{X_0}}\leq e^{kB}e^{-\varphi},
\]
and hence
\[
 \mathcal I(\varphi)\subset\mathcal I(\Phi|_{X_0}).
\]
Performing the construction on each of the finitely many connected
components considered at the beginning and combining the resulting
functions completes the proof in general.
\end{proof}

\begin{cor}\label{cor: family mid}
Let $p:X\rightarrow\Delta$ be a proper holomorphic submersion from a
K\"ahler manifold $X$, assume that $K_{X_t}$ is nef for every
$t\in\Delta$, and let $m\geq2$. Let $\alpha$ be a smooth real
$d$-closed $(1,1)$-form representing
$c_1((m-1)K_{X/\Delta})$, and let
$\varphi\in\operatorname{PSH}(X_0,\alpha|_{X_0})$. Then, for every
$0<r<1$, there exists
$\Phi\in\operatorname{PSH}(p^{-1}(\Delta_{1-r}),\alpha)$ such that
\[
 \varphi\leq\Phi|_{X_0}+O(1),
 \qquad
 \mathcal I(\varphi)\subset\mathcal I(\Phi|_{X_0}).
\]
Via the trivialization of $K_\Delta$ by $dt$, the same statement may be
written with $(m-1)K_X$ in place of $(m-1)K_{X/\Delta}$.
\end{cor}
\begin{proof}[Proof of Corollary \ref{cor: family mid}]
Apply Theorem \ref{thm: family mid} with $k=m-1$.
\end{proof}

\begin{rem}
Theorem \ref{thm: family mid} uses two features special to the relative
canonical class: P\u aun's total-space positivity
\eqref{eq:paun-total-positive} and the normalization identity
\eqref{eq:fiber-normalization-inv}. It does not imply the analogous
extension statement for an arbitrary nef line bundle and an arbitrary
submanifold.
\end{rem}

\section{Invariance of plurigenera for K\"ahler families}

\begin{thm}\label{thm: inv-Kah}
Let $p:X\rightarrow\Delta$ be a proper holomorphic submersion from a
K\"ahler manifold $X$, and assume that $K_{X_t}$ is nef for every
$t\in\Delta$. Then, for every integer $m\geq1$, every $t_0\in\Delta$, and
every $s\in H^0(X_{t_0},mK_{X_{t_0}})$, there exists
$\widetilde s\in H^0(X,mK_X)$ whose restriction to $X_{t_0}$, under the
identification induced by the standard coordinate on $\Delta$, is $s$.
Consequently, $P_m(X_t)$ is independent of $t\in\Delta$.
\end{thm}

\begin{proof}
We first argue componentwise. As in the first paragraph of the proof of
Theorem~\ref{thm: family mid}, every connected component of $X$ maps onto
$\Delta$, and only finitely many components occur. Both the extension
problem and the spaces $H^0(X_t,mK_{X_t})$ split as finite direct sums over
these open-and-closed components. We may therefore assume that $X$ is
connected. The Ehresmann--Stein-factorization argument used there then shows
that every fiber $X_t$ is connected.

It is enough to treat the fiber $X_0$. Indeed, for an arbitrary
$t_0\in\Delta$ we can replace $p$ by $\chi\circ p$, where $\chi$ is an
automorphism of $\Delta$ satisfying $\chi(t_0)=0$. The two fiber
identifications differ only by the nonzero scalar $(\chi'(t_0))^m$, which
can be removed by rescaling the resulting extension.

Let $t$ denote the standard coordinate on $\Delta$, set
$t_X:=t\circ p$, and write $dt_X:=p^*(dt)$. Since $dt$ trivializes
$K_\Delta$, we have
\[
 K_X\simeq K_{X/\Delta}\otimes p^*K_\Delta
       \simeq K_{X/\Delta}.
\]
On the central fiber, the adjunction isomorphism trivialized by the
conormal form $dt_X$ is
\[
 mK_{X_0}\xrightarrow{\sim}(mK_X)|_{X_0},
 \qquad u\longmapsto u\otimes(dt_X)^{\otimes m}.
\]

Suppose first that $m=1$. The function
\[
 \rho=-\log(1-|t_X|^2)
\]
is a smooth plurisubharmonic exhaustion of $X$: its sublevel sets are
compact because $p$ is proper. Thus $X$ is weakly pseudoconvex and carries
the given K\"ahler metric. The holomorphic function $t_X:X\to\mathbb C$
is bounded and $0$ is a regular value. Moreover, $X_0$ is compact and
$dt_X$ is nowhere zero along $X_0$. Consequently, for the trivial metric
on $L=\mathcal O_X$,
\[
 \int_{X_0}\frac{|s\otimes dt_X|_\omega^2}
 {|dt_X|_\omega^2}\,dV_{X_0,\omega}<+\infty.
\]
Thus $s\otimes dt_X\in H^0(X_0,K_X|_{X_0})$ satisfies the integrability
hypothesis of Theorem \ref{thm: ext}. Applying that theorem with the trivial
line bundle gives $G\in H^0(X,K_X)$ such that
$G|_{X_0}=s\otimes dt_X$. Under the preceding identification, $G$ is the
required extension of $s$.

Assume from now on that $m\geq2$. If $s=0$, there is nothing to prove,
so assume that $s\neq0$. Let $h_\omega$ be the smooth metric on
$K_{X/\Delta}$ induced by a K\"ahler form $\omega$ on $X$, and let
\[
 \alpha:=\Theta_{h_\omega^{m-1}}((m-1)K_{X/\Delta}).
\]
Define on $X_0$
\[
 \varphi_s:=\frac{m-1}{m}\log|s|^2_{h_\omega^m}.
\]
To check the curvature without fixing a normalization of the
Poincar\'e--Lelong formula, choose a local holomorphic frame $e$ of
$K_{X/\Delta}$, write $h_\omega(e,e)=e^{-\Psi}$ and
$s=f e^{\otimes m}$. Then
\[
 \varphi_s=\frac{m-1}{m}\log|f|^2-(m-1)\Psi
\]
and hence
\[
 \alpha|_{X_0}+\sqrt{-1}\partial\bar\partial\varphi_s
 =\frac{m-1}{m}\sqrt{-1}\partial\bar\partial\log|f|^2\geq0.
\]
Thus $\varphi_s$ is $\alpha|_{X_0}$-psh and has analytic singularities.
Moreover,
\[
 |f|^2e^{-\varphi_s}
 =|f|^{2/m}e^{(m-1)\Psi}\in L^1_{\mathrm{loc}},
\]
so $f\in\mathcal I(\varphi_s)$.

Fix $0<R<1$ and set
\[
 \Omega_R:=p^{-1}(\Delta_R).
\]
Theorem \ref{thm: family mid}, applied with $k=m-1$ and $r=1-R$,
gives an $\alpha$-psh function $\Phi$ on $\Omega_R$ such that
\[
 \mathcal I(\varphi_s)\subset\mathcal I(\Phi|_{X_0}).
\]
Set $L:=(m-1)K_{X/\Delta}$ and equip $L$ with the singular Hermitian
metric
\[
 h_L:=h_\omega^{m-1}e^{-\Phi}.
\]
Its curvature current is
\[
 \Theta_{h_L}(L)
 =\alpha+\sqrt{-1}\partial\bar\partial\Phi\geq0.
\]

We verify all hypotheses of Theorem \ref{thm: ext}. The function
\[
 \rho_R=-\log(R^2-|t_X|^2)
\]
is a smooth plurisubharmonic exhaustion of $\Omega_R$, since $p$ is
proper. Thus $\Omega_R$ is weakly pseudoconvex, and it is K\"ahler with
respect to $\omega|_{\Omega_R}$. The defining map is the bounded
holomorphic function $t_X:\Omega_R\to\mathbb C$, and $0$ is a regular
value. The section to be extended is
\[
 f_0:=s\otimes dt_X
 \in H^0\!\left(X_0,(K_{\Omega_R}\otimes L)|_{X_0}\right).
\]
In the local notation above, the only singular factor in its squared norm
is $|f|^2e^{-\Phi}$. Since
$f\in\mathcal I(\varphi_s)\subset\mathcal I(\Phi|_{X_0})$, this factor is
locally integrable. Since $X_0$ is compact and $dt_X$ is nowhere zero on
$X_0$, all remaining smooth factors, including
$|dt_X|_\omega^{-2}$, are uniformly bounded. Therefore
\[
 \int_{X_0}
 \frac{|f_0|_{h_L}^2}{|dt_X|_\omega^2}\,dV_{X_0,\omega}<+\infty.
\]
Theorem \ref{thm: ext} yields
\[
 G_R\in H^0(\Omega_R,K_{\Omega_R}\otimes L),
 \qquad G_R|_{X_0}=f_0.
\]
Using $K_X=K_{X/\Delta}\otimes p^*K_\Delta$, multiplication by the
nowhere-vanishing base form gives
\[
 \widetilde s_R:=G_R\otimes(dt_X)^{\otimes(m-1)}
 \in H^0(\Omega_R,mK_X).
\]
Its restriction is $s\otimes(dt_X)^{\otimes m}$, hence corresponds to
$s$ under the fixed identification.

It remains to pass from this neighborhood extension to a section on all of
$X$. By Grauert's proper direct-image theorem,
\[
 \mathscr E_m:=p_*(mK_X)
\]
is a coherent analytic sheaf on $\Delta$. The section $\widetilde s_R$
defines a germ $\sigma\in(\mathscr E_m)_0$. Since $\Delta$ is Stein,
Cartan's theorem A gives global sections
$G_1,\ldots,G_N\in H^0(\Delta,\mathscr E_m)=H^0(X,mK_X)$ whose germs
generate $(\mathscr E_m)_0$. Thus
\[
 \sigma=\sum_{j=1}^N a_j(G_j)_0
\]
for some germs $a_j\in\mathscr O_{\Delta,0}$. Set
\[
 \widetilde s:=\sum_{j=1}^N a_j(0)G_j\in H^0(X,mK_X).
\]
Then
\[
 \sigma-\widetilde s_0
 =\sum_{j=1}^N\bigl(a_j-a_j(0)\bigr)(G_j)_0
 \in\mathfrak m_0(\mathscr E_m)_0,
\]
The natural restriction morphism
\[
 (\mathscr E_m)_0\longrightarrow
 H^0\!\left(X_0,(mK_X)|_{X_0}\right)
\]
is $\mathscr O_{\Delta,0}$-linear when the target is viewed via evaluation
at $0$; hence it annihilates $\mathfrak m_0(\mathscr E_m)_0$. Therefore the
difference restricts to zero on $X_0$, and
$\widetilde s|_{X_0}$ corresponds to $s$.

We have now proved the extension assertion for both $m=1$ and $m\geq2$.
We finally prove constancy of the plurigenus. Fix $t_0\in\Delta$ and let
$s_1,\ldots,s_N$ be a basis of $H^0(X_{t_0},mK_{X_{t_0}})$. By the
extension statement just proved, they admit global extensions
$\widetilde s_1,\ldots,\widetilde s_N\in H^0(X,mK_X)$. Their restrictions
remain linearly independent for $t$ near $t_0$. Indeed, otherwise there
would exist $t_\nu\to t_0$ and vectors
$c^{(\nu)}=(c_1^{(\nu)},\ldots,c_N^{(\nu)})$ of norm one such that
\[
 \sum_{j=1}^N c_j^{(\nu)}\widetilde s_j|_{X_{t_\nu}}=0.
\]
After passing to a subsequence, $c^{(\nu)}\to c$ with $\|c\|=1$.
For each $x\in X_{t_0}$, choose a local holomorphic section of $p$ through
$x$ and a local frame of $mK_X$. Evaluating the preceding identities along
that section and passing to the limit gives
\[
 \sum_{j=1}^N c_j s_j(x)=0.
\]
Since $x$ is arbitrary, this contradicts the linear independence of the
$s_j$. Consequently
\[
 P_m(X_t)\geq P_m(X_{t_0})
\]
for $t$ near $t_0$. Since $p$ is a submersion, it is flat, and the locally
free sheaf $mK_{X/\Delta}$ is flat over $\Delta$. Grauert's upper
semicontinuity theorem therefore gives the reverse inequality near $t_0$.
Thus $P_m(X_t)$ is locally constant, and hence
constant because $\Delta$ is connected.
\end{proof}

	\end{document}